\documentclass{amsart}

\usepackage[english]{babel}
\usepackage{hyperref}

\usepackage[utf8]{inputenc}
\usepackage[T1]{fontenc}
\usepackage{pifont}

\usepackage[margin=3cm]{geometry}

\usepackage{amsmath}
\usepackage{amssymb}
\usepackage{amsthm}
\usepackage{amscd}
\usepackage{mathrsfs}
\usepackage{hyperref}
\usepackage[abbrev,backrefs]{amsrefs}

\usepackage{tikz-cd}

\usepackage{graphicx}

\usepackage {cancel}

\numberwithin{equation}{section}

\newtheorem{theorem}{Theorem}[section]	
\newtheorem{proposition}[theorem]{Proposition}	
\newtheorem{corollary}[theorem]{Corollary}	
\newtheorem{lemma}[theorem]{Lemma}

\newtheorem{remark}[theorem]{Remark}

\newcommand{\C}{\mathscr{C}}
\newcommand{\T}{\mathscr{T}}
\newcommand{\F}{\mathscr{F}}
\newcommand{\z}{\mathscr{Z}}

\newcommand{\s}{\mathscr{S}}
\newcommand{\MonPos}{\mathsf{Mon_{\geq 0}}}

\begin{document}

\title{\textsc{The stable category of preordered groups}}
\author{Aline Michel}
\thanks{The author's research is funded by a FRIA doctoral grant of the \emph{Communaut\'e française de Belgique}}
\email{aline.michel@uclouvain.be}
\address{Universit{\'e} Catholique de Louvain, Institut de Recherche en Math{\'e}matique et Physique, Chemin du Cyclotron 2, 1348 Louvain-la-Neuve, Belgium}
\date{}
\begin{abstract}
In this article, we present the stable category of preordered groups associated with some $\z$-pretorsion theory. We first define such a category as well as the related functor, and then study their properties. By doing so, we provide a description of both $\z$-kernels and $\z$-cokernels in the category of preordered groups. Finally, we prove the universal property of the stable category.   
\end{abstract}

\maketitle

\section{Introduction}

A \emph{preordered group} $(G,\le)$ is a group $G = (G,+,0)$ endowed with a preorder relation $\le$ which is compatible with the addition $+$ of $G$, in the sense that, if $a \le c$ and $b \le d$ for $a$, $b$, $c$ and $d$ in $G$, then $a + b \le c + d$. Preordered groups are the objects of a category, the category $\mathsf{PreOrdGrp}$ of preordered groups, whose arrows are the preorder preserving group morphisms. It turns out that this category is isomorphic to another one whose description is as follows:
\begin{itemize}
\item the objects are pairs $(G,P_G)$ (often depicted by means of an inclusion arrow $P_G \rightarrowtail G$) where $G$ is a group and $P_G$ a submonoid of $G$ closed under conjugation in $G$, which is called the \emph{positive cone} of $G$;
\item the arrows $(G,P_G) \rightarrow (H,P_H)$ are given by group morphisms $f \colon G \rightarrow H$ such that $f(P_G) \subseteq P_H$. In this sense, such an arrow can be seen as a pair $(f,\bar{f})$ making the following square commute:
\begin{equation} \label{arrow in PreOrdGrp}
\begin{tikzcd}
P_G \arrow[r,"{\bar{f}}"] \arrow[d,tail]
& P_H \arrow[d,tail]\\
G \arrow[r,"f"']
& H.
\end{tikzcd}
\end{equation}
\end{itemize}

In \cite{CMFM}, some important categorical properties of preordered groups were investigated. In that article, Clementino, Martins-Ferreira and Montoli proved, among other things, that $\mathsf{PreOrdGrp}$ is both complete and cocomplete, and \emph{normal} in the sense of \cite{JZ10}, but neither Barr-exact nor protomodular. The kernel of an arrow $(f,\bar{f})$ as in \eqref{arrow in PreOrdGrp} is given by $((K,P_K),(k,\bar{k}))$ with $k \colon K \rightarrow G$ the kernel of $f$ in the category $\mathsf{Grp}$ of groups and $\bar{k} \colon P_K \rightarrow P_G$ the kernel of $\bar{f}$ in the category $\mathsf{Mon}$ of monoids or, alternatively, $P_K = K \cap P_G$. In order to compute the cokernel of $(f,\bar{f})$, we first need to take the cokernel $(Q,q)$ of $f$ in $\mathsf{Grp}$. The positive cone of $Q$ is then given by $P_Q = q(P_H)$. It was also observed that an arrow $(f,\bar{f})$ as in \eqref{arrow in PreOrdGrp} is a monomorphism precisely when $f$ is injective. It is an epimorphism when $f$ is surjective, and a regular epimorphism when moreover $\bar{f}$ is surjective.

In the paper \cite{GM}, some exactness properties of the category $\mathsf{PreOrdGrp}$ of preordered groups were then studied. In particular, $\mathsf{PreOrdGrp}$ was proven to admit a \emph{pretorsion theory}. For the reader's convenience, let us quickly recall the general definition of this last notion.

A \emph{$\z$-pretorsion theory} \cite{FF,FFG} in an arbitrary category $\C$ is given by a pair $(\T,\F)$ of full and replete subcategories of $\C$, with $\z = \T \cap \F$, such that the following two conditions hold:
\begin{itemize}
\item[(PT1)] any arrow from any $T \in \T$ to any $F \in \F$ factors through an object of $\z$;
\item[(PT2)] for any object $C$ in $\C$, there exists a \emph{short $\z$-exact sequence}
\begin{center}
\begin{tikzcd}
T \arrow[r,"k"]
& C \arrow[r,"q"]
& F
\end{tikzcd}
\end{center}
with $T \in \T$ and $F \in \F$.
\end{itemize}
By a \emph{short $\z$-exact sequence}, we mean a pair $(k,q)$ of composable morphisms in $\C$ such that $k$ is a \emph{$\z$-kernel} of $q$, i.e.
\begin{itemize}
\item $q \cdot k$ factors through an object of $\z$;
\item for any arrow $\alpha \colon X \rightarrow C$ such that $q \cdot \alpha$ factors through an object of $\z$, there exists a unique morphism $\phi \colon X \rightarrow T$ satisfying the identity $k \cdot \phi = \alpha$,
\end{itemize}
and $q$ is a \emph{$\z$-cokernel} of $k$, i.e. 
\begin{itemize}
\item $q \cdot k$ factors through an object of $\z$;
\item for any arrow $\beta \colon C \rightarrow Y$ such that $\beta \cdot k$ factors through an object of $\z$, there exists a unique morphism $\psi \colon F \rightarrow Y$ satisfying the identity $\psi \cdot q = \beta$.
\end{itemize}
\begin{center}
\begin{tikzcd}
T \arrow[rr,"k"]
& 
& C \arrow[rr,"q"] \arrow[dr,"{\beta}"']
&
& F \arrow[dl,dotted,"{\psi}"]\\
& X \arrow[ur,"{\alpha}"'] \arrow[ul,dotted,"{\phi}"]
&
& Y
&
\end{tikzcd}
\end{center}
Note that, when the category $\C$ is pointed and the class $\z$ is reduced to the zero object $0$ of $\C$, we then recover the classical definitions of kernels and cokernels. In this case, a $\z$-pretorsion theory is nothing but a \emph{torsion theory} (see \cite{BG} or \cite{CDT}, for instance). If there is a $\z$-pretorsion theory $(\T,\F)$ in a category $\C$, then the subcategory $\T$ is called a \emph{torsion subcategory} of $\C$ and $\F$ a \emph{torsion-free subcategory}.

As proved in \cite{GM}, there is a pretorsion theory in the category $\mathsf{PreOrdGrp}$ of preordered groups. The torsion subcategory is the category $\mathsf{Grp(PreOrd)}$ of \emph{internal groups in the category of preordered sets}. Its objects are given by those preordered groups whose preorder is moreover symmetric (i.e. an equivalence relation) or, equivalently, by those pairs $(G,P_G)$ for which the positive cone $P_G$ is a group. The torsion-free subcategory, on the other hand, is given by the full subcategory $\mathsf{ParOrdGrp}$ of \emph{partially ordered groups}, that is, of preordered groups endowed with an antisymmetric preorder. These partially ordered groups correspond, via the isomorphism of categories mentioned above, to the pairs $(G,P_G)$ whose positive cone $P_G$ is a \emph{reduced monoid} (which means that the only element in $P_G$ having also its inverse in $P_G$ is the neutral element $0$ of $G$).

\begin{proposition} \label{Z-pretorsion theory} 
The pair $(\mathsf{Grp(PreOrd)},\mathsf{ParOrdGrp})$ of full replete subcategories of $\mathsf{PreOrdGrp}$ is a $\z$-pretorsion theory in $\mathsf{PreOrdGrp}$, where $\z = \mathsf{Grp(PreOrd)} \cap \mathsf{ParOrdGrp}$ is given by
\[\z = \left\{(G,P_G) \in \mathsf{PreOrdGrp} \mid P_G = 0\right\},\]
i.e. by the class of preordered groups endowed with the discrete order.
\end{proposition}

The reader can refer to \cite{GM} for a detailed proof of this Proposition. In particular, it is proved in \cite{GM} that the \emph{canonical short $\z$-exact sequence} associated with any preordered group $(G,P_G)$ in the pretorsion theory $(\mathsf{Grp(PreOrd)},\mathsf{ParOrdGrp})$ (i.e. the short $\z$-exact sequence of the second point (PT2) of the definition of a pretorsion theory) is given by
\begin{center}
\begin{tikzcd}
N_G \arrow[r,tail] \arrow[d,tail]
& P_G \arrow[r,two heads,"{\bar{\eta}_G}"] \arrow[d,tail]
& \eta_G(P_G) \arrow[d,tail]\\
G \arrow[r,equal]
& G \arrow[r,two heads,"{\eta_G}"']
& G/N_G
\end{tikzcd}
\end{center}
where $N_G = \{x \in G \mid x \in P_G \ \text{and} \ -x \in P_G\}$ is a normal subgroup of $G$ and $\eta_G \colon G \twoheadrightarrow G/N_G$ the canonical quotient in $\mathsf{Grp}$.

Preordered groups in the class $\z$ are called \emph{$\z$-trivial objects}, and morphisms in $\mathsf{PreOrdGrp}$ that factor through a $\z$-trivial object are then called \emph{$\z$-trivial morphisms}. As shown in \cite{GM}, $\z$-trivial morphisms in this context can be characterized as follows:

\begin{proposition} \label{trivial morphisms}
A morphism $(f,\bar{f}) \colon (G,P_G) \rightarrow (H,P_H)$ in $\mathsf{PreOrdGrp}$ is $\z$-trivial if and only if $\bar{f} = 0$.
\end{proposition}
Indeed, if $(f,\bar{f})$ is $\z$-trivial then, by definition, $(f,\bar{f})$ factors through a $\z$-trivial object $(A,0)$:
\begin{center}
\begin{tikzcd}[row sep=small,column sep=small]
P_G \arrow[rr,"\bar{f}"] \arrow[dr] \arrow[ddd,tail]
& & P_H \arrow[ddd,tail]\\
 & 0 \arrow[d,tail] \arrow[ur] & \\
 & A \arrow[dr] & \\
G \arrow[ur] \arrow[rr,"f"']
& & H.
\end{tikzcd}
\end{center}
In particular $\bar{f} = 0$. Conversely, consider in the category $\mathsf{Grp}$ of groups the (regular epi, mono)-factorization of $f$. Then, by assumption, the following diagram commutes:
\begin{center}
\begin{tikzcd}
P_G \arrow[rr,"\bar{f}"] \arrow[dr] \arrow[ddd,tail]
& & P_H \arrow[ddd,tail]\\
 & 0 \arrow[d,tail] \arrow[ur] & \\
 & f(G) \arrow[dr,tail] & \\
G \arrow[ur,two heads] \arrow[rr,"f"']
& & H.
\end{tikzcd}
\end{center}
Accordingly, $(f,\bar{f})$ factors through $(f(G),0) \in \z$, i.e. it is a $\z$-trivial morphism.

The purpose of the article is to answer the following question: what is the ``universal'' torsion theory associated with the above pretorsion theory? More precisely, we need to find a  pointed category $\s$, which will be called the \emph{stable category} of $\mathsf{PreOrdGrp}$, containing a torsion theory $(\T,\F)$, as well as a functor $\Phi$ from $\mathsf{PreOrdGrp}$ to $\s$ which is a \emph{torsion theory functor}, i.e.
\begin{enumerate}
\item $\Phi(G,P_G) \in \T$ for any $(G,P_G) \in \mathsf{Grp(PreOrd)}$ and $\Phi(G,P_G) \in \F$ for any $(G,P_G) \in \mathsf{ParOrdGrp}$;
\item for any preordered group $(G,P_G)$, $\Phi$ sends the canonical short $\z$-exact sequence associated with $(G,P_G)$ in the pretorsion theory $(\mathsf{Grp(PreOrd)},\mathsf{ParOrdGrp})$ to a short exact sequence in $\s$.
\end{enumerate}
In particular, this functor $\Phi$ must send any $\z$-trivial object to the zero object in the stable category $\s$ and then any $\z$-trivial morphism to the zero morphism.

\begin{remark}
\emph{Such a construction was initially realized by Facchini and Finocchiaro in \cite{FF}, in the context of preordered sets. Borceux, Campanini and Gran then extended this construction to the setting of internal preorders in a \emph{coherent category} (see articles \cite{BCG} and \cite{BCG22}). 
In the more recent article \cite{BCG22'}, the results were then extended to the context of \emph{lextensive categories}. The context and universal property considered in the present article are nevertheless very different from those mentioned above.
}
\end{remark}

\section{Definition of the stable category of $\mathsf{PreOrdGrp}$}

In order to define the stable category of $\mathsf{PreOrdGrp}$, we first recall two well-known results about monoids, already discussed in \cite{CMFM}.

\begin{lemma} \label{lemma1 Mon}
If a monoid $M$ is embeddable in a group $G$, then it is embeddable in its \emph{group completion} (also called \emph{Grothendieck group}) $grp(M)$. 
\end{lemma}

\begin{proof}
Consider the canonical morphism $j \colon M \rightarrow grp(M)$ as well as the injection $i \colon M \rightarrowtail G$ to the group $G$:
\begin{center}
\begin{tikzcd}
M \arrow[rr,"j"] \arrow[dr,tail,"i"']
& & grp(M) \arrow[dl,dotted,"\phi"]\\
 & G. &
\end{tikzcd}
\end{center}
By the universal property of the Grothendieck group of $M$, there exists a unique group homomorphism $\phi \colon grp(M) \rightarrow G$ such that $\phi \cdot j = i$ in the category $\mathsf{Mon}$ of monoids. Since $i$ is a monomorphism, it then easily follows that also $j$ is a monomorphism, hence $M$ is embeddable in $grp(M)$.
\end{proof}

\begin{lemma} \label{lemma2 Mon}
Let $M$ be a monoid which is embeddable in a group. Then, $M$ is the positive cone of a group $G$ if and only if, for any $a, b \in M$, there exist $x, y \in M$ such that $a + b = b + x = y + a$.
\end{lemma}

\begin{proof}
If $M$ is the positive cone of a group $G$, then, for any $a, b \in M$, $-b + a + b \in M$ and $a + b - a \in M$ (since $M$ is closed under conjugation in $G$). By taking $x = -b + a + b$ and $y = a + b - a$, we can then check that $a + b = b + x = y + a$.

Assuming now that this last condition holds, let us prove that $M$ is the positive cone of its group completion $grp(M)$. Let $m \in M$ and let $x \in grp(M)$, that is, $x = x_1 - x_2 + \cdots + x_{n-1} - x_n$ for $x_1,\cdots,x_n \in M$. Then,
\[x + m - x = x_1 - x_2 + \cdots + x_{n-1} - x_n + m + x_n - x_{n-1} + \cdots + x_2 - x_1.\]
We observe that $-x_n + m + x_n \in M$. Indeed, by assumption, there exist $a_1, b_1 \in M$ such that $m + x_n = x_n + a_1 = b_1 + m$ so that, in particular, $-x_n + m + x_n = a_1 \in M$. We next notice that $x_{n-1} + a_1 - x_{n-1} \in M$ since, by hypothesis, $x_{n-1} + a_1 = a_1 + b_2 = a_2 + x_{n-1}$ for some $a_2, b_2 \in M$, which implies that $x_{n-1} + a_1 - x_{n-1} = a_2 \in M$, that is, $x_{n-1} - x_n + m + x_n - x_{n-1} \in M$. We then proceed by induction and show that $x + m - x \in M$. This proves that $M$ is closed under conjugation in $grp(M)$. Since $M$ is embeddable in a group, it is therefore embeddable in $grp(M)$ thanks to Lemma \ref{lemma1 Mon}. As a conclusion, $(grp(M),M) \in \mathsf{PreOrdGrp}$. 
\end{proof}

Consider then the full subcategory $\MonPos$ of the category $\mathsf{Mon}$ of monoids, whose objects are all monoids $M$ which are embeddable in a group and such that, for any $a, b \in M$, there exist $x, y \in M$ for which the identity $a + b = b + x = y + a$ holds. 

Consider also the functor $P \colon \mathsf{PreOrdGrp} \rightarrow \MonPos$ defined, for any preordered group $(G,P_G)$, by $P(G,P_G) = P_G$ and, for any morphism $(f,\bar{f})$ of preordered groups, by $P(f, \bar{f}) = \bar{f}$. This is the \emph{positive cone functor}. Note that, thanks to Lemmas \ref{lemma1 Mon} and \ref{lemma2 Mon}, this functor is actually well-defined.

We will prove in the subsequent sections that the category $\MonPos$ is the stable category of $\mathsf{PreOrdGrp}$.

\section{$\z$-kernels and $\z$-cokernels in $\mathsf{PreOrdGrp}$}

In this section, we present an explicit description of $\z$-kernels and $\z$-cokernels in the category $\mathsf{PreOrdGrp}$ of preordered groups.

\subsection{Description of $\z$-kernels and $\z$-cokernels in $\mathsf{PreOrdGrp}$}

\begin{proposition} \label{Z-kernels}
Let $(f,\bar{f}) \colon (G,P_G) \rightarrow (H,P_H)$ be a morphism in $\mathsf{PreOrdGrp}$. Then, a $\z$-kernel of $(f,\bar{f})$ is given by the morphism 
\[(k,\bar{k}) \colon (G,P_G \cap \mathsf{Ker}(f)) \rightarrow (G,P_G)\]
where $k$ is the identity morphism $1_G$. Note that $P_G \cap \mathsf{Ker}(f) = \mathsf{Ker}(\bar{f})$ in the category $\mathsf{Mon}$ of monoids.
\end{proposition}

\begin{proof}
We first compute that $\bar{f} \cdot \bar{k} = 0$ since $\bar{k}$ is the kernel of $\bar{f}$ in $\mathsf{Mon}$. It then follows that $(f,\bar{f}) \cdot (k,\bar{k})$ is a $\z$-trivial morphism thanks to Proposition \ref{trivial morphisms}.

We next consider a morphism $(\alpha,\bar{\alpha}) \colon (A,P_A) \rightarrow (G,P_G)$ in $\mathsf{PreOrdGrp}$ such that $(f,\bar{f}) \cdot (\alpha,\bar{\alpha})$ is $\z$-trivial, i.e. such that $\bar{f} \cdot \bar{\alpha} = 0$ thanks to Proposition \ref{trivial morphisms}.
\begin{center}
\begin{tikzcd}
\mathsf{Ker}(\bar{f}) \arrow[rr,tail,"\bar{k}"] \arrow[ddd,tail,"\psi"']
& & P_G \arrow[rr,"\bar{f}"] \arrow[ddd,tail,"g"]
& & P_H \arrow[ddd,tail]\\
 & P_A \arrow[d,tail,"a"] \arrow[ur,"\bar{\alpha}"] \arrow[ul,dotted,"\bar{\phi}"']
& & & \\
 & A \arrow[dl,dotted,"\phi"] \arrow[dr,"\alpha"']
& & & \\
G \arrow[rr,equal,"{k = 1_G}"']
& & G \arrow[rr,"f"']
& & H
\end{tikzcd}
\end{center}
By the universal property of kernels in $\mathsf{Mon}$ there exists a unique morphism $\bar{\phi} \colon P_A \rightarrow \mathsf{Ker}(\bar{f})$ such that $\bar{k} \cdot \bar{\phi} = \bar{\alpha}$. Now, we put $\phi = \alpha$ since it is the only possible group morphism $\phi \colon A \rightarrow G$ with $k \cdot \phi = \alpha$. We moreover compute that
\[\phi \cdot a = k \cdot \phi \cdot a = \alpha \cdot a = g \cdot \bar{\alpha} = g \cdot \bar{k} \cdot \bar{\phi} = k \cdot \psi \cdot \bar{\phi} = \psi \cdot \bar{\phi},\]
which means that $(\phi,\bar{\phi}) \colon (A,P_A) \rightarrow (G,\mathsf{Ker}(\bar{f}))$ is a morphism in $\mathsf{PreOrdGrp}$. It is obviously unique with the property $(k,\bar{k}) \cdot (\phi,\bar{\phi}) = (\alpha,\bar{\alpha})$.
\end{proof}

\begin{proposition} \label{Z-cokernels}
Let $(f,\bar{f}) \colon (G,P_G) \rightarrow (H,P_H)$ be a morphism in $\mathsf{PreOrdGrp}$ and denote by $S$ the normal closure of $\bar{f}(P_G)$ in $H$, i.e.
\[S = \{s_1 - s_2 + \cdots - s_n \mid s_i = h_i + f(g_i) - h_i, \ \text{with} \ h_i \in H \ \text{and} \ g_i \in P_G, \ \text{for any} \ i \in \{1,\cdots,n\}\}.\]
If we write $q \colon H \twoheadrightarrow H/S$ for the quotient morphism, then the arrow
\[(q,\bar{q}) \colon (H,P_H) \twoheadrightarrow (H/S,q(P_H))\]
is a $\z$-cokernel of $(f,\bar{f})$.
\end{proposition}

\begin{proof}
For any $x \in P_G$, $f(x) \in \bar{f}(P_G) \subseteq S$, so that $q(f(x)) = 0$ since $S$ is the kernel of $q$ in the category $\mathsf{Grp}$ of groups. This implies that $\bar{q} \cdot \bar{f} = 0$, which means that $(q,\bar{q}) \cdot (f,\bar{f})$ is a $\z$-trivial morphism.

Consider then a morphism $(\alpha,\bar{\alpha}) \colon (H,P_H) \rightarrow (A,P_A)$ in $\mathsf{PreOrdGrp}$ such that $(\alpha,\bar{\alpha}) \cdot (f,\bar{f})$ is $\z$-trivial, that is, such that $\bar{\alpha} \cdot \bar{f} = 0$.
\begin{center}
\begin{tikzcd}
P_G \arrow[rrr,"\bar{f}"] \arrow[dr,two heads] \arrow[ddd,tail]
& & & P_H \arrow[ddd,tail] \arrow[rr,two heads,"\bar{q}"] \arrow[dr,"\bar{\alpha}"]
& & q(P_H) \arrow[dl,dotted,"\bar{\phi}"'] \arrow[ddd,tail,"\psi"]\\
 & \bar{f}(P_G) \arrow[urr,tail] \arrow[dr,tail]
& & & P_A \arrow[d,tail,"a"]
& \\
 & & S \arrow[dr,tail,"t"] 
& & A 
& \\
G \arrow[rrr,"f"']
& & & H \arrow[ur,"\alpha"'] \arrow[rr,two heads,"q"']
& & H/S \arrow[ul,dotted,"\phi"]
\end{tikzcd}
\end{center}
Let $s \in S$. Then, there exist $h_i \in H$ and $g_i \in P_G$ such that 
\[s = s_1 - s_2 + \cdots - s_n\]
where $s_i = h_i + f(g_i) - h_i$ for any $i \in \{1,\cdots,n\}$. One then computes that, for any $i \in \{1,\cdots,n\}$, 
\[\alpha(s_i) = \alpha(h_i) + (\alpha \cdot f)(g_i) - \alpha(h_i) = \alpha(h_i) + 0 - \alpha(h_i) = 0\]
since $\bar{\alpha} \cdot \bar{f} = 0$, so that $\alpha(s) = \alpha(s_1) - \alpha(s_2) + \cdots - \alpha(s_n) = 0$. Since $q$ is the cokernel of $t \colon S \rightarrowtail H$ in the category $\mathsf{Grp}$ of groups, there is a unique group morphism $\phi \colon H/S \rightarrow A$ such that $\phi \cdot q = \alpha$. Now, seeing that $\bar{q}$ is a regular epimorphism in the category $\mathsf{Mon}$ of monoids and that $a \colon P_A \rightarrow A$ is a monomorphism, the universal property of strong epimorphisms yields a unique arrow $\bar{\phi} \colon q(P_H) \rightarrow P_A$ in $\mathsf{Mon}$ such that $\bar{\phi} \cdot \bar{q} = \bar{\alpha}$ and $a \cdot \bar{\phi} = \phi \cdot \psi$. Accordingly, there exists a unique morphism $(\phi,\bar{\phi}) \colon (H/S,q(P_H)) \rightarrow (A,P_A)$ in $\mathsf{PreOrdGrp}$ such that $(\phi,\bar{\phi}) \cdot (q,\bar{q}) = (\alpha,\bar{\alpha})$. The uniqueness of such an arrow comes from the fact that $(q,\bar{q})$ is an epimorphism in $\mathsf{PreOrdGrp}$.
\end{proof}

It is then easy to prove the following:

\begin{corollary} \label{canonical short z-exact sequence}
For any preordered group $(G,P_G)$, the sequence
\begin{center}
\begin{tikzcd}
N_G \arrow[r,tail] \arrow[d,tail]
& P_G \arrow[r,two heads,"{\bar{\eta}_G}"] \arrow[d,tail]
& \eta_G(P_G) = P_G/N_G \arrow[d,tail]\\
G \arrow[r,equal]
& G \arrow[r,two heads,"{\eta_G}"']
& G/N_G
\end{tikzcd}
\end{center}
is a short $\z$-exact sequence in $\mathsf{PreOrdGrp}$. It is the \emph{canonical short $\z$-exact sequence} associated with $(G,P_G)$ in the pretorsion theory $(\mathsf{Grp(PreOrd)},\mathsf{ParOrdGrp})$.
\end{corollary}

\begin{remark}
\emph{As observed in \cite{GM}, the arrow $\bar{\eta}_G \colon P_G \twoheadrightarrow \eta_G(P_G)$ is actually the cokernel of $N_G \rightarrowtail P_G$, that is, $\eta_G(P_G) = P_G/N_G$. 
}
\end{remark}

\subsection{$\z$-kernels and $\z$-cokernels as special pullbacks and pushouts in $\mathsf{PreOrdGrp}$}\hspace*{-0.1em}

We then notice that these $\z$-kernels and $\z$-cokernels can be seen as some particular pullbacks and pushouts, respectively, as is the case for the classical notions of kernels and cokernels. To do this, we first recall some results from the second section of the article \cite{GJM} by Grandis, Janelidze and M\'arki.

In their paper, Grandis, Janelidze and M\'arki consider a full subcategory $\C_0$ of a category $\C_1$, which is simultaneously reflective and coreflective, i.e. a diagram 
\begin{center}
\begin{tikzcd}
\C_1 \arrow[r,bend left=4em,"D"] \arrow[r,bend right=4em,"C"']
& \C_0 \arrow[l,"E"']
\end{tikzcd}
\end{center}
where $C$, $D$ and $E$ are three functors as follows:
\begin{itemize}
\item $E$ is fully faithful;
\item $D$ is a right adjoint left inverse of $E$;
\item $C$ is a left adjoint left inverse of $E$.
\end{itemize}
Denote by $\iota$ the counit of the adjunction $E \dashv D$ and by $\pi$ the unit of the adjunction $C \dashv E$, and assume, for simplicity, that $\C_0$ is a full replete subcategory of $\C_1$ and that $E$ is the inclusion functor.

A category $\C_1$ with the above data can be seen as a category equipped with an \emph{ideal of null morphisms} (which is generated by $\C_0$, in the sense that an arrow is in the ideal if and only if it factors through an object of $\C_0$). A \emph{semiexact category} is then a category equipped with a closed ideal, where any morphism has both a kernel and a cokernel relatively to this ideal.

The link between these ``relative'' kernels and cokernels on the one hand, and the usual limits and colimits on the other hand is as follows:

\begin{proposition} \label{relative kernels as pbs} \cite{GJM}
Let $f \colon A \rightarrow B$ be a morphism in $\C_1$ for which $\iota_B$ is a monomorphism. Then, the following conditions are equivalent:
\begin{enumerate}
\item an arrow $k \colon K \rightarrow A$ is a ``relative'' kernel of $f \colon A \rightarrow B$;
\item there exists a unique morphism $g \colon K \rightarrow D(B)$ making the diagram
\begin{center}
\begin{tikzcd}
K \arrow[r,dotted,"g"] \arrow[d,"k"']
& D(B) \arrow[d,tail,"{\iota_B}"]\\
A \arrow[r,"f"']
& B
\end{tikzcd}
\end{center}
a pullback.
\end{enumerate}
\end{proposition}

\begin{proposition} \label{relative cokernels as pos} \cite{GJM}
Let $f \colon A \rightarrow B$ be a morphism in $\C_1$ for which $\pi_A$ is an epimorphism. Then, the following conditions are equivalent:
\begin{enumerate}
\item an arrow $q \colon B \rightarrow Q$ is a ``relative'' cokernel of $f \colon A \rightarrow B$;
\item there exists a unique morphism $g \colon C(A) \rightarrow Q$ making the diagram
\begin{center}
\begin{tikzcd}
A \arrow[r,two heads,"{\pi_A}"] \arrow[d,"f"']
& C(A) \arrow[d,dotted,"g"]\\
B \arrow[r,"q"']
& Q
\end{tikzcd}
\end{center}
a pushout.
\end{enumerate}
\end{proposition}

Let $\z$ denote the full subcategory of $\mathsf{PreOrdGrp}$ whose objects are all $\z$-trivial objects. In this case, a kernel (respectively a cokernel) with respect to the ideal generated by $\z$ is nothing but a $\z$-kernel (respectively a $\z$-cokernel). So, let us take $\C_0 = \z$ and $\C_1 = \mathsf{PreOrdGrp}$ in the theory developed by Grandis, Janelidze and M\'arki. The functor $E$ is then the inclusion functor and we define the two functors $C$ and $D$ in the following way: for any preordered group $(G,P_G)$,
\begin{center}
\begin{tikzcd}
\mathsf{PreOrdGrp} \arrow[r,bend left=4em,"D"] \arrow[r,bend right=4em,"C"']
& \z \arrow[l,hook',"E"']
\end{tikzcd}
\end{center}
\begin{itemize}
\item $D(G,P_G) = (G,0)$;
\item $C(G,P_G) = (G/M_G,0)$
where $M_G = \{x_1 - x_2 + \cdots + x_{n-1} - x_n \mid x_i \in P_G \ \forall 1 \le i \le n\}$ is the normal closure of the submonoid $P_G$ in $G$.
\end{itemize}

Let us show that we then have the following chain of adjunctions: $C \dashv E \dashv D$.

\begin{proposition}
The functor $D \colon \mathsf{PreOrdGrp} \rightarrow \z$ is a right adjoint left inverse of the inclusion functor $E \colon \z \hookrightarrow \mathsf{PreOrdGrp}$.
\end{proposition}

\begin{proof}
It is first of all clear that $D$ is a left inverse of $E$. Indeed, for any $\z$-trivial object $(G,0)$, $(D \cdot E)(G,0) = D(G,0) = (G,0)$. 

We are now going to prove that, for any preordered group $(G,P_G)$, the arrow $(\iota_G,\bar{\iota}_G) = (1_G,0) \colon ED(G,P_G) = (G,0) \rightarrow (G,P_G)$ is the $(G,P_G)$-component of the counit $\iota$ of the adjunction $E \dashv D$.
\begin{center}
\begin{tikzcd}
0 \arrow[rr,"{\bar{\iota}_G = 0}"] \arrow[dd,tail]
& & P_G \arrow[dd,tail]\\
 & 0 \arrow[ul,dotted,"\bar{\psi} = 0" description] \arrow[dd,tail] \arrow[ur,"{\bar{\phi}} = 0"']
& \\
G \arrow[rr,near end,"{\iota_G = 1_G}"]
& & G \\
 & A \arrow[ul,dotted,"{\psi = \phi}" description] \arrow[ur,"{\phi}"']
& 
\end{tikzcd}
\end{center}
Let $(A,0) \in \z$ and consider an arrow $(\phi,\bar{\phi}) \colon (A,0) \rightarrow (G,P_G)$. It suffices to take $\psi = \phi$ (since $\phi$ is the unique morphism such that $1_G \cdot \phi = \phi$) and $\bar{\psi} = 0$ (since of course $\bar{\iota}_G \cdot 0 = 0 = \bar{\phi}$). The pair $(\psi,\bar{\psi})$ is then clearly a morphism of preordered groups, unique with the property $(\iota,\bar{\iota}) \cdot (\psi,\bar{\psi}) = (\phi,\bar{\phi})$. We conclude that $D$ is a right adjoint for $E$.
\end{proof}

\begin{proposition}
The functor $C \colon \mathsf{PreOrdGrp} \rightarrow \z$ is a left adjoint left inverse of the inclusion functor $E \colon \z \hookrightarrow \mathsf{PreOrdGrp}$.
\end{proposition}

\begin{proof}
We first check that $C$ is a left inverse of $E$. For any $(G,0) \in \z$, $M_G = 0$, so that $(C \cdot E)(G,0) = C(G,0) = (G/0,0) = (G,0)$, as desired.

Let us now show that, for any preordered group $(G,P_G)$, the $(G,P_G)$-component of the unit $\pi$ of the adjunction $C \dashv E$ is given by the morphism $(\pi_G,\bar{\pi}_G) \colon (G,P_G) \rightarrow (G/M_G,0)$, where $\pi_G \colon G \twoheadrightarrow G/M_G$ is the canonical quotient.
\begin{center}
\begin{tikzcd}
 & P_G \arrow[rr,"{\bar{\pi}_G = 0}"] \arrow[dr,"{\bar{\phi} = 0}"'] \arrow[dd,tail] \arrow[ddl,tail]
& & 0 \arrow[dl,dotted,"{\bar{\psi} = 0}" description] \arrow[dd,tail]\\
 & & 0 \arrow[dd,tail] & \\
M_G \arrow[r,tail,"i"']
& G \arrow[rr,two heads,near end,"{\pi_G}"'] \arrow[dr,"{\phi}"']
& & G/M_G \arrow[dl,dotted,"{\psi}"]\\
 & & A &
\end{tikzcd}
\end{center}
Consider a morphism $(\phi,\bar{\phi}) \colon (G,P_G) \rightarrow (A,0)$ (where $(A,0)$ is then a $\z$-trivial object). For any $m \in M_G$, we can compute that $\phi(m) = 0$. Indeed, $m = x_1 - x_2 + \cdots + x_{n-1} - x_n$ for some $x_i \in P_G$ ($i \in \{1,\cdots,n\}$), which implies that
\begin{align*}
\phi(m) & = \phi(x_1) - \phi(x_2) + \cdots + \phi(x_{n-1}) - \phi(x_n)\\
& = \bar{\phi}(x_1) - \bar{\phi}(x_2) + \cdots + \bar{\phi}(x_{n-1}) - \bar{\phi}(x_n)\\
& = 0
\end{align*}
since $\bar{\phi} = 0$. This means that $\phi \cdot i = 0$ where $i \colon M_G \rightarrowtail G$ is the inclusion of $M_G$ in $G$. By the universal property of cokernels, there exists a unique morphism $\psi \colon G/M_G \rightarrow A$ such that $\psi \cdot \pi_G = \phi$. By taking $\bar{\psi} = 0$, we then get a morphism of preordered groups $(\psi,\bar{\psi}) \colon (G/M_G,0) \rightarrow (A,0)$, unique with the property $(\psi,\bar{\psi}) \cdot (\pi_G,\bar{\pi}_G) = (\phi,\bar{\phi})$. This proves that $C$ is a left adjoint for $E$.
\end{proof}

According to Propositions \ref{relative kernels as pbs} and \ref{relative cokernels as pos}, our $\z$-kernels and $\z$-cokernels are then actually some special pullbacks and pushouts.

\section{Properties of the stable category $\MonPos$ and of the functor $P$}

Let us start with two obvious observations:

\begin{proposition}
The category $\MonPos$ is pointed with zero object the trivial monoid $0$.
\end{proposition}

\begin{proposition} \hspace*{1em}
\begin{enumerate}
\item For any $\z$-trivial object $(G,P_G)$ in $\mathsf{PreOrdGrp}$, $P(G,P_G) = 0$ in $\MonPos$.
\item For any $\z$-trivial morphism $(f, \bar{f})$ in $\mathsf{PreOrdGrp}$, $P(f,\bar{f})$ is the zero morphism in $\MonPos$.
\end{enumerate}
\end{proposition}

It is now useful to note that there is a torsion theory in the category $\MonPos$. The torsion-free subcategory is the full subcategory $\mathsf{Red}\MonPos$ of $\MonPos$ whose objects are in addition reduced monoids, while the torsion subcategory is the one generated by the objects in $\MonPos$ which are also groups. Notice that this latter subcategory then clearly corresponds to the category $\mathsf{Grp}$ of groups.

\begin{proposition} \label{TT in MonPos}
The pair $(\mathsf{Grp},\mathsf{Red}\MonPos)$ of full replete subcategories of $\MonPos$ is a torsion theory in $\MonPos$.
\end{proposition}

\begin{proof}
Let us first prove that any morphism $f \colon A \rightarrow B$ in $\MonPos$, where $A \in \mathsf{Grp}$ and $B \in \mathsf{Red}\MonPos$, is the zero morphism. For any $x \in A$, we have that $-x \in A$ (since $A$ is a group), so that both $f(x)$ and $f(-x) = -f(x)$ are in $B$. This implies that $f(x) = 0$ since $B$ is reduced. 

Let now $M \in \MonPos$, and consider the subgroup 
\[U(M) = \{x \in M \mid -x \in M\} \in \mathsf{Grp}\]
of invertible elements of $M$. Define also, for any $a, b \in M$, the following equivalence relation: 
\begin{align*}
a \sim b & \quad \Leftrightarrow \quad \exists u \in U(M) \ \text{such that} \ a = b + u\\
& \quad \Leftrightarrow \quad \exists v \in U(M) \ \text{such that} \ b = a + v.
\end{align*}
It is actually a congruence. Indeed, let $a, b, c, d \in M$ be such that $a \sim b$ and $c \sim d$. Then, there exists $u \in U(M)$ such that $a = b + u$ and there exists $u' \in U(M)$ such that $c = d + u'$. Accordingly,
\[a + c = b + u + d + u' = (b + d) + (-d + u + d + u').\]
Let us prove that $-d + u +d \in U(M)$. Since $M \in \MonPos$, there exist $x, y \in M$ such that $u + d = d + x = y + u$. In particular, $-d + u + d = x \in M$. On the other hand, there exist $z, w \in M$ such that $-u + d = d + z = w - u$, so that $-d - u + d = z \in M$, i.e. $-(-d + u + d) \in M$. This shows that $-d + u + d \in U(M)$, and then that $-d + u + d + u' \in U(M)$, as desired. As a consequence, $a + c \sim b + d$ and $\sim$ is a congruence on $M$. Let us then consider the monoid $M/U(M) := M/\sim$ and let us check that it is reduced. Let $y \in M/U(M)$ such that also $-y \in M/U(M)$. Then, there exist $x$ and $z$ in $M$ such that $\bar{x} = y$ and $\bar{z} = -y$ (where $\bar{x}$ denotes the equivalence class of $x$ with respect to $\sim$). We compute that $\overline{x + z} = \bar{x} + \bar{z} = y - y = 0$, which implies that $x + z \in U(M)$. In particular, $-(x + z) \in M$. From this, we deduce that 
\[-x = z - z - x = z - (x + z) \in M,\]
i.e. that $x \in U(M)$. Accordingly, $y = \bar{x} = 0$, hence $M/U(M) \in \mathsf{Red}\MonPos$. As a conclusion, we have built a short exact sequence 
\begin{center} 
\begin{tikzcd}
U(M) \arrow[r,"{\kappa_M}"]
& M \arrow[r,two heads,"{\eta_M}"]
& M/U(M)
\end{tikzcd} 
\end{center}
in $\MonPos$ with $U(M) \in \mathsf{Grp}$ and $M/U(M) \in \mathsf{Red}\MonPos$.
\end{proof}

\begin{proposition}
The functor $P \colon \mathsf{PreOrdGrp} \rightarrow \MonPos$ is a torsion theory functor.
\end{proposition}

\begin{proof}
It is first of all clear that $P(G,P_G) \in \mathsf{Grp}$ (respectively $\in \mathsf{Red}\MonPos$) for any $(G,P_G) \in \mathsf{Grp(PreOrd)}$ (respectively $\in \mathsf{ParOrdGrp}$). We next easily observe that, for any preordered group $(G,P_G)$, if we apply the functor $P$ to the canonical short $\z$-exact sequence
\begin{center}
\begin{tikzcd}
N_G \arrow[r,tail] \arrow[d,tail]
& P_G \arrow[r,two heads,"{\bar{\eta}_G}"] \arrow[d,tail]
& P_G/N_G \arrow[d,tail]\\
G \arrow[r,equal]
& G \arrow[r,two heads,"{\eta_G}"']
& G/N_G
\end{tikzcd}
\end{center}
associated with $(G,P_G)$ in the pretorsion theory $(\mathsf{Grp(PreOrd)},\mathsf{ParOrdGrp})$ (see Corollary \ref{canonical short z-exact sequence}), we then get 
\begin{center}
\begin{tikzcd}
N_G = U(P_G) \arrow[r,tail]
& P_G \arrow[r,two heads]
& P_G/N_G
\end{tikzcd}
\end{center}
which is a short exact sequence in $\MonPos$.
\end{proof}

Moreover, it is easy to see that the positive cone functor $P \colon \mathsf{PreOrdGrp} \rightarrow \MonPos$ preserves a specific form of short exact sequences.

\begin{proposition}
The functor $P \colon \mathsf{PreOrdGrp} \rightarrow \MonPos$ preserves short exact sequences of the following form:
\begin{equation} \label{special SES}
\begin{tikzcd}
P_G \arrow[r,equal] \arrow[d,tail]
& P_G \arrow[r,two heads] \arrow[d,tail]
& 0 \arrow[d,tail]\\
H \arrow[r,tail]
& G \arrow[r,two heads]
& G/H
\end{tikzcd}
\end{equation}
for any $(G,P_G) \in \mathsf{PreOrdGrp}$ such that $H$ is a normal subgroup of $G$ containing the positive cone $P_G$ of $G$.
\end{proposition}

\section{The universal property of the stable category}

In this section, we prove that the torsion theory ($\mathsf{Grp}$, $\mathsf{Red}\MonPos$) in $\mathsf{\MonPos}$ is the ``best'' one can associate with the pretorsion theory ($\mathsf{Grp(PreOrd)}$, $\mathsf{ParOrdGrp}$) in $\mathsf{PreOrdGrp}$, i.e. we show that it is the \emph{universal} torsion theory associated with ($\mathsf{Grp(PreOrd)}$, $\mathsf{ParOrdGrp}$). 

\begin{theorem} \label{UP of the stable category}
Let $\C$ be a pointed category containing a torsion theory $(\T,\F)$. Let $F \colon \mathsf{PreOrdGrp} \rightarrow \C$ be a functor satisfying the following properties:
\begin{enumerate}
\item $F$ is a torsion theory functor;
\item $F$ preserves short exact sequences of the form \eqref{special SES}.
\end{enumerate}
Then, there exists a unique functor $\hat{F} \colon \MonPos \rightarrow \C$ such that the triangle
\begin{center}
\begin{tikzcd}
\mathsf{PreOrdGrp} \arrow[rr,"P"] \arrow[dr,"F"']
& & \MonPos \arrow[dl,dotted,"{\hat{F}}"]\\
 & \C &
\end{tikzcd}
\end{center}
commutes (up to isomorphism).
\end{theorem}

\begin{proof}
Let $M \in \MonPos$. Then, $(grp(M),M) \in \mathsf{PreOrdGrp}$. Let us therefore define $\hat{F}$ on objects as follows: $\hat{F}(M) = F(grp(M),M)$ for any $M \in \MonPos$. It now remains to prove that $F(grp(M),M) \cong F(G,M)$ for any group $G$ such that $(G,M) \in \mathsf{PreOrdGrp}$. This will then prove that $\hat{F} \cdot P \cong F$ on objects. Suppose we have a preordered group $(G,M)$. By the universal property of the Grothendieck group $grp(M)$, there exists a unique group morphism $\alpha \colon grp(M) \rightarrow G$ making the square 
\begin{center}
\begin{tikzcd}
M \arrow[r,equal,"{\bar{\alpha} = 1_M}"] \arrow[d,tail]
& M \arrow[d,tail]\\
grp(M) \arrow[r,dotted,"{\alpha}"']
& G
\end{tikzcd}
\end{center}
commute, which entails that $(\alpha,\bar{\alpha}) \colon (grp(M),M) \rightarrow (G,M)$ is a morphism of preordered groups. We now observe that 
\begin{center}
\begin{tikzcd}
M \arrow[r,equal,"{\bar{\alpha}=1_M}"] \arrow[d,tail]
& M \arrow[r,two heads,"{\bar{q}}"] \arrow[d,tail]
& 0 \arrow[d,tail]\\
grp(M) \arrow[r,tail,"{\alpha}"']
& G \arrow[r,two heads,"q"']
& G/grp(M)
\end{tikzcd}
\end{center}
is a short exact sequence in $\mathsf{PreOrdGrp}$. Indeed, since the monoid $M$ is closed under conjugation in $G$, its group completion $grp(M)$ is then a normal subgroup of $G$. By the assumption $(2)$ on $F$, if we apply the functor $F$ to this sequence, we then obtain a short exact sequence in $\C$. In particular, $F(\alpha, \bar{\alpha}) = \mathsf{ker} (F(q,\bar{q}))$. Notice now that $F(q,\bar{q})$ is the zero morphism in $\C$. Indeed, according to the assumption $(1)$ on the functor $F$, $(q,\bar{q})$ being a $\z$-trivial morphism in $\mathsf{PreOrdGrp}$ is sent by $F$ to the zero morphism. As a conclusion, $F(\alpha,\bar{\alpha})$ is the kernel of its cokernel which is $0$, that is, it is an isomorphism. We deduce that $F(grp(M),M) \cong F(G,M)$.

The definition of $\hat{F} \colon \MonPos \rightarrow \C$ on arrows is now easy. Let $\bar{f} \colon M \rightarrow N$ be a morphism in $\MonPos$. Thanks to the universal property of the Grothendieck group $grp(M)$, there exists a unique group morphism $g \colon grp(M) \rightarrow grp(N)$ such that the square
\begin{center}
\begin{tikzcd}
M \arrow[r,"{\bar{f}}"] \arrow[d,tail]
& N \arrow[d,tail]\\
grp(M) \arrow[r,dotted,"g"']
& grp(N)
\end{tikzcd}
\end{center}
commutes, which means that $(g,\bar{f}) \colon (grp(M),M) \rightarrow (grp(N),N)$ is a morphism of preordered groups. Let us then define $\hat{F}(\bar{f})$ as follows: $\hat{F}(\bar{f}) = F(g,\bar{f})$ where $g$ is the above induced group morphism. It remains to check that $F(g,\bar{f}) \cong F(f,\bar{f})$ for any group morphism $f \colon G \rightarrow H$ such that $(f,\bar{f})$ is a morphism of preordered groups. By the universal property of the Grothendieck groups $grp(M)$ and $grp(N)$, there exist a unique group morphism $\sigma \colon grp(M) \rightarrow G$ and a unique group morphism $\tau \colon grp(N) \rightarrow H$ such that $\sigma \cdot i_M = i_G$ and $\tau \cdot i_N = i_H$:
\begin{center}
\begin{tikzcd}
 & M \arrow[rr,"{\bar{f}}"] \arrow[dl,equal] \arrow[dd,tail,near start,"{i_G}"]
& & N \arrow[dl,equal] \arrow[dd,tail,"{i_H}"]\\
M \arrow[rr,near end,"{\bar{f}}"'] \arrow[dd,tail,"{i_M}"']
& & N \arrow[dd,tail,near end,"{i_N}"']
& \\
 & G \arrow[rr,near start,"f"]
& & H \\
grp(M) \arrow[ur,dotted,"{\sigma}"] \arrow[rr,"g"']
& & grp(N). \arrow[ur,dotted,"{\tau}"']
& 
\end{tikzcd}
\end{center}
We observe that, for any element $x_1 - x_2 + \cdots + x_{n-1} - x_n$ of $grp(M)$,
\begin{align*}
(f \cdot \sigma)&(x_1 - x_2 + \cdots + x_{n-1} - x_n) \\
& = f(x_1 - x_2 + \cdots + x_{n-1} - x_n) \\
& = f(x_1) - f(x_2) + \cdots + f(x_{n-1}) - f(x_n) \\
& = g(x_1) - g(x_2) + \cdots + g(x_{n-1}) - g(x_n) \\
& = g(x_1 - x_2 + \cdots + x_{n-1} - x_n) \\
& = (\tau \cdot g)(x_1 - x_2 + \cdots + x_{n-1} - x_n)
\end{align*}
since $f$ and $g$ coincide on $M$, which means that $f \cdot \sigma = \tau \cdot g$.
In other words, there exist two isomorphisms $F(\sigma,1_M)$ and $F(\tau,1_N)$ in $\C$ (see the first part of the proof) such that $F(f,\bar{f}) \cdot F(\sigma,1_M) = F(\tau,1_N) \cdot F(g,\bar{f})$, which proves that $F(g,\bar{f}) \cong F(f,\bar{f})$. 

Finally, the uniqueness of $\hat{F}$ such that $\hat{F} \cdot P \cong F$ follows from the surjectivity on objects of the positive cone functor $P$.
\end{proof}

Note that the induced functor $\hat{F} \colon \MonPos \rightarrow \C$ of Theorem \ref{UP of the stable category} is itself a torsion theory functor (as is the case for $F$ and $P$):

\begin{proposition}
The induced functor $\hat{F} \colon \MonPos \rightarrow \C$ of Theorem \ref{UP of the stable category} is a torsion theory functor.
\end{proposition}

\begin{proof}
We first easily check that 
\begin{itemize}
\item for any $M \in \mathsf{Grp}$, $\hat{F}(M) = F(grp(M),M) \in \T$ since $(grp(M),M) \in \mathsf{Grp(PreOrd)}$ and $F$ is a torsion theory functor;
\item for any $M \in \mathsf{Red}\MonPos$, $\hat{F}(M) = F(grp(M),M) \in \F$ since $(grp(M),M) \in \mathsf{ParOrdGrp}$ and $F$ is a torsion theory functor. 
\end{itemize}

Consider now the canonical short exact sequence associated with any $M \in \MonPos$ in the torsion theory ($\mathsf{Grp}$, $\mathsf{Red}\MonPos$) (see Proposition \ref{TT in MonPos}):
\begin{equation} \label{hatF is a TTF diag1}
\begin{tikzcd}
U(M) \arrow[r,"{\kappa_M}"]
& M \arrow[r,two heads,"{\eta_M}"]
& M/U(M).
\end{tikzcd}
\end{equation}
Then,
\begin{equation} \label{hatF is a TTF diag2}
\begin{tikzcd}[column sep = large]
U(M) \arrow[r,tail,"{\kappa_M}"] \arrow[d,tail]
& M \arrow[r,two heads,"{\eta_M}"] \arrow[d,tail]
& M/U(M) \arrow[d,tail]\\
grp(M) \arrow[r,equal]
& grp(M) \arrow[r,two heads,"{\eta_{grp(M)}}"']
& grp(M)/U(M) 
\end{tikzcd}
\end{equation}
is the canonical short $\z$-exact sequence associated with the preordered group $(grp(M),M)$ in the pretorsion theory ($\mathsf{Grp(PreOrd)}$, $\mathsf{ParOrdGrp}$) (see Corollary \ref{canonical short z-exact sequence}). Since $F$ is a torsion theory functor, it sends this short $\z$-exact sequence in $\mathsf{PreOrdGrp}$ to a short exact sequence in $\C$. Now, applying the functor $\hat{F}$ to the short exact sequence \eqref{hatF is a TTF diag1} amounts to applying the functor $F$ to the short $\z$-exact sequence \eqref{hatF is a TTF diag2}, so that the short exact sequence \eqref{hatF is a TTF diag1} is sent by $\hat{F}$ to a short exact sequence in $\C$. 
\end{proof}

\vspace*{1cm}

\paragraph{\textbf{Comment.} The content of this paper is part of the author's PhD Thesis, defended in June 2023 at the \emph{Université catholique de Louvain} \cite{MA}. }

\paragraph{\textbf{Acknowledgements.} The author warmly thanks Professor George Janelidze for the fruitful discussions they had on the subject of the article in March 2023.}

\end{document}